\documentclass{article}
\usepackage[left=3cm,right=3cm,top=2.5cm,bottom=2.5cm]{geometry}
\usepackage{algorithm}
\usepackage{algpseudocode}
\usepackage{amsmath} 
\usepackage{amsthm}
\usepackage{amsfonts}
\usepackage{amssymb}
\usepackage{bbm}
\usepackage{relsize}
\usepackage{units}
\usepackage{xcolor}
\usepackage[shortlabels, inline]{enumitem}

\usepackage{todonotes}

\usepackage{algorithm}
\usepackage{algorithmicx}
\usepackage{algpseudocode}

\usepackage{graphicx}
\graphicspath{{/home/andreas/Documents/FORSCHUNG/Forschung_github/DeepTVprivate/python/1d/test_boundedweights/denoising/layers_64-128-64_acts_relu-relu-relu_l2_0/ruleAR/learning_rate0.01/}}

\usepackage{subcaption} 
\usepackage{adjustbox} 

\usepackage{pgfplots}
\pgfplotsset{compat=newest} 

\usepackage{hyperref}
\usepackage[
    capitalise,
    noabbrev,
  ]{cleveref}
\usepackage{tikz}
\usetikzlibrary{arrows}

\usepackage{xparse}
\newtheorem{theorem}{Theorem}[section]
\newtheorem{example}[theorem]{Example}
\newtheorem{proposition}[theorem]{Proposition}
\newtheorem{remark}[theorem]{Remark}
\newtheorem{corollary}[theorem]{Corollary}

\newtheorem{lemma}[theorem]{Lemma}

\newcommand{\N}{\mathbb{N}}

\newcommand{\R}{\mathbb{R}}

\newcommand{\mcD}{\mathcal{D}}

\newcommand{\mcH}{\mathcal{H}}
\newcommand{\mcE}{\mathcal{E}}
\newcommand{\mcR}{\mathcal{R}}
\newcommand{\mcJ}{\mathcal{J}}
\newcommand{\mcB}{\mathcal{B}}
\newcommand{\mcA}{\mathcal{A}}
\newcommand{\mcL}{\mathcal{L}}
\newcommand{\mcU}{\mathcal{U}}
\newcommand{\mcV}{\mathcal{V}}

\newcommand{\Tr}{\mathrm{Tr}}
\newcommand{\mcHreg}[1]{\mcH_{\mathrm{sm},#1}}
\newcommand{\overmcHreg}[1]{\overline{\mcH}_{\mathrm{sm},#1}}

\newcommand{\dx}{\, \text{d}x}

\newcommand{\dz}{\, \text{d}z}
\newcommand{\ds}{\text{d}\Gamma}
\DeclareMathOperator*{\argmin}{arg\,min}

\def\FD{\mathrm{FD}}
\def\DR{\mathrm{DR}}
\def\reg{\mathrm{reg}}

\DeclareDocumentCommand{\InputData}{mO{}}{
  \def\DataPrefix{#2}
  \input{#1}}
  
 \usepackage{fancyhdr}
\title{Non-Uniqueness of Solutions in Neural Variational Methods}

\author{Andreas Langer\thanks{Corresponding author}\hspace{0.15cm}$^{, }$\thanks{Center for Mathematical Sciences, Lund University, Box 118, 221 00 Lund, Sweden, \url{andreas.langer@math.lth.se}}
}
\date{}

\begin{document}
\maketitle
\begin{abstract}
Recent work has shown that strong-form physics-informed neural networks (PINNs) based on pointwise enforcement of differential operators can be ill-posed due to the combination of sufficiently expressive neural network trial spaces with finitely many measurements.

In this work, we develop an abstract analytical framework that isolates this finite-information mechanism and extends its applicability beyond strong-form formulations. We apply the framework to three representative variational neural discretizations: the Deep Ritz method, neural network discretizations of variational regularization functionals, and weak PINNs. Despite their differing formulations, these methods constrain the neural trial function only through finitely many linear measurements, such as quadrature evaluations or finite-dimensional test spaces.

We show that this structural feature leads to ill-posed discrete optimization problems, manifested by non-uniqueness or degeneracy of minimizers, independently of the well-posedness of the underlying continuous variational problem.
\end{abstract}

\section{Introduction}
Neural network-based discretizations have become increasingly popular for the numerical solution of partial differential equations (PDEs) and variational problems. In these approaches, the unknown solution is represented by a neural network, and the resulting problem is treated numerically through a discretized loss functional.
This paradigm encompasses a range of methods, including physics-informed neural networks (PINNs), where differential equations are enforced in strong form at collocation points via pointwise residuals \cite{LaLiFo:98, RaPeKa:19}; weak or variational PINNs (wPINNs), which employ finite-dimensional test spaces and enforce the weak form of the governing equations \cite{DeMiMo:2024, Khodayi-MehrZavlanos:2020}; and the Deep Ritz method, which directly minimizes variational energy functionals over neural network trial spaces using pointwise derivatives \cite{EYu:18}.
In practice, when pointwise derivatives arise in such formulations, they are typically computed via automatic differentiation, which has enabled their efficient implementation in modern software frameworks.
Further related neural network discretization approaches include finite-difference PINNs (FD-PINNs), where differential operators are discretized via finite differences prior to training \cite{lim2022physics}, as well as neural network discretizations of variational regularization problems, where classical regularization functionals such as total variation or higher-order energies are discretized and evaluated on neural network ansatz functions \cite{LangerBehnamian:24}.

Despite their rapid adoption, neural network-based discretizations, most prominently PINNs, have been reported to exhibit practical limitations. In particular, PINNs are known to struggle when confronted with complex geometries, intricate boundary conditions, high-frequency solution components, or pronounced multiscale phenomena, often failing to recover accurate solutions \cite{raissi2018deep,fuks2020limitations, krishnapriyan2021characterizing, wang2022and}. 
By contrast, considerably less is known about the structural properties of variational neural formulations, including wPINNs and energy-based methods such as the Deep Ritz method. While these approaches are used in practice, a rigorous analysis of the well-posedness of the resulting discretized optimization problems does not yet seem to be available in general. 
For the Deep Ritz method, existing evidence points to possible difficulties already at the discrete level. In particular, numerical studies have reported instabilities in discretized formulations \cite{CourteZeinhofer2021}, and \cite{RiveraTaylorOmellaPardo2022} shows that quadrature errors can lead to misleading discrete loss values and poor approximations despite apparently successful optimization. Moreover, \cite{RiveraTaylorOmellaPardo2022} discusses several possible remedies, including Monte Carlo integration, adaptive integration, polynomial approximation, and regularization. 
Weak PINNs, referred to as variational PINNs in \cite{BeCaPi2022}, are examined there from a Petrov--Galerkin perspective, and instability and ill-posedness of weak formulations at the discrete level are reported. Motivated by discrete inf-sup considerations, stabilization is achieved by projecting the residual onto carefully chosen polynomial test spaces, yielding existence and uniqueness for specific trial-test pairings. This analysis, however, is formulation- and discretization-dependent and does not provide a general structural characterization of neural variational discretizations.

More recently, an analytical framework identifying structural ill-posedness of the discrete optimization problems arising from neural network-based PDE discretizations has been developed for strong-form PINNs and FD-PINNs \cite{Langer:26}. 
There, it is shown that strong-form PINNs based on pointwise differentiation give rise to ill-posed optimization problems admitting infinitely many solutions even when the underlying continuous PDE is well-posed. Importantly, the analysis also clarifies that replacing pointwise derivatives by finite-difference operators alters the structure of the discrete problem and restores classical discretization behavior on the stencil.
The results of \cite{Langer:26} reveal a structural mechanism underlying this ill-posedness, namely the combination of an infinite-dimensional neural trial class with loss terms that probe the solution only through finitely many measurements. However, that analysis is carried out entirely in the context of strong-form formulations and does not address variational discretizations, weak formulations, or energy-based methods. 

In this work, we develop an abstract analytical framework that isolates this finite-information mechanism and allows its systematic study independently of the particular PDE formulation. We then demonstrate how this framework applies to three representative variational neural discretizations: the Deep Ritz method, neural network discretizations of variational regularization functionals, and wPINNs. While these methods differ in formulation, they share the same structural feature that pointwise derivatives are combined with losses that depend only on finitely many measurements, such as quadrature points or finite-dimensional test spaces.

We show that this finite-information structure leads to ill-posedness of the resulting discrete optimization problems, manifested by non-uniqueness and the resulting degeneracy of the loss landscape, permitting minimizers arbitrarily far from the continuous solution, even when the underlying continuous variational problem is well-posed. 
The three settings considered in this paper thus serve as representative examples illustrating how the same mechanism identified for strong-form PINNs reappears in variational and weak formulations, though with different manifestations. In this sense, the present work provides a unified perspective on ill-posedness phenomena in neural discretizations of PDEs and variational problems.

The remainder of the paper is organized as follows. In \cref{sec:notation-prelim}, we introduce the analytical framework, neural network hypothesis classes, and the notion of finite-measurement loss functionals, which form the technical basis of our analysis. 
\Cref{Sec:DRM} investigates the Deep Ritz method and shows that quadrature combined with pointwise differentiation generically leads to ill-posed discrete optimization problems admitting infinitely many minimizers. 
In \cref{Sec:Regularization}, we extend this analysis to neural network discretizations of variational regularization problems and demonstrate that discretizations based on pointwise derivatives fail to enforce regularization at the discrete level; a comparison with finite-difference discretizations of the regularization term is also provided. 
\Cref{Sec:wPINN} is devoted to weak PINNs, where we show that ill-posedness arises already from the finite-dimensional test space, independently of numerical quadrature. 
Finally, we conclude in \cref{Sec:Conclusion} by discussing the implications of our findings.

\section{Preliminaries}\label{sec:notation-prelim}

In this section, we introduce notation and basic analytical concepts needed for the subsequent analysis. We fix the functional-analytic framework, define the neural network hypothesis classes, and introduce technical tools that will be used in the proofs of the main results.

\subsection{Analytical Framework}

Throughout the paper, we fix a spatial dimension $d\in\N$ and a bounded open domain $\Omega\subset\R^d$ with Lipschitz boundary $\partial\Omega$.
We write $\Gamma\subseteq \partial\Omega$ for a (measurable) portion of the boundary on which essential boundary conditions are prescribed.
Unless stated otherwise, all functions are understood to be defined on $\Omega$ and to be scalar-valued, i.e., $u:\Omega\to\R$.
We use $\Omega^h\subset\Omega$ to denote a finite set of interior points and $\Gamma^h\subset\Gamma$ a finite set of boundary points. We define their union by $Z^h:=\Omega^h\cup\Gamma^h$. For any finite set $A$, we write $|A|$ for its cardinality. A functional depending only on finitely many pointwise evaluations and finitely many of its derivatives at points in $\Omega^h$ and $\Gamma^h$ will be called \emph{finite-measurement} in the sense made precise in \cref{Sec:FiniteLoss}.

Let $1\le p\le \infty$. We write $|\cdot|_p$ for the standard $\ell^p$-norm on finite-dimensional Euclidean spaces (in particular on $\R^d$). Further, we denote by $L^p(\Omega)$ the usual Lebesgue spaces equipped with the norms $\|\cdot\|_{L^p(\Omega)}$, and by $W^{m,p}(\Omega)$ the Sobolev spaces of functions whose weak partial derivatives up to order $m\in\N_0$ belong to $L^p(\Omega)$, equipped with the norm \(\|u\|_{W^{m,p}(\Omega)} := \sum_{|\beta|\le m} \|D^\beta u\|_{L^p(\Omega)},\) where $\beta=(\beta_1,\dots,\beta_d)\in\N_0^d$ is a multi-index, $|\beta|:=\beta_1+\cdots+\beta_d$ denotes its order, and $D^\beta u$ denotes the corresponding weak derivative. We write $\Tr(u)$ for the trace of $u\in W^{1,p}(\Omega)$ on $\partial\Omega$.

Let $U$ be a Banach space. A functional $E:U\to(-\infty,\infty]$ is called \emph{lower semicontinuous} if \(E(u)\le \liminf_{n\to\infty} E(u_n)\) for every sequence $(u_n)_n\subset U$ converging to $u$ in the norm topology. It is called \emph{weakly lower semicontinuous} if the same implication holds for every sequence $(u_n)_n$ converging weakly to $u$ in $U$. We call $E$ \emph{coercive on $U$} if $E(u_n)\to\infty$ whenever $\|u_n\|_U\to\infty$. In the setting of the direct method of the calculus of variations, weak lower semicontinuity together with coercivity and weak closedness of the admissible set yields existence of minimizers, provided bounded sequences admit weakly convergent subsequences (e.g., in reflexive spaces).
Alternatively, lower semicontinuity together with compactness of the admissible set implies existence of minimizers.

\subsection{Neural networks}\label{Sec:NN}

Throughout this work, we consider fully connected feedforward neural networks, which we simply call \emph{neural networks}. All hidden layers employ the same scalar activation function \(\sigma:\mathbb{R}\to\mathbb{R}\), applied componentwise, while the output layer is linear. The network depth is denoted by \(L\in\mathbb{N}\), with input dimension \(d_0=d\) and output dimension \(d_L=1\). Such a network defines a mapping \(f:\mathbb{R}^{d_0}\to\mathbb{R}^{d_L}\) via the recursive relations
\begin{equation*}
\begin{aligned}
&\varphi_0(x)=x,\\
&\varphi_i(x)=\sigma\!\left(W_i\varphi_{i-1}(x)+b_i\right), \qquad i=1,\dots,L-1,\\
&f(x)=W_L\varphi_{L-1}(x)+b_L,
\end{aligned}
\end{equation*}
where \(W_i\in\mathbb{R}^{d_i\times d_{i-1}}\) and \(b_i\in\mathbb{R}^{d_i}\) are the weight matrices and bias vectors, respectively, and \((d_0,\dots,d_L)\) specifies the layer widths.

Let \(\mathcal{N}_L\) denote the collection of all depth-\(L\) neural networks mapping \(\mathbb{R}^d\) to \(\mathbb{R}\), allowing arbitrary widths in the hidden layers.
We then introduce the hypothesis class
\begin{equation*}
\mathcal{H} := \{\, f|_{{\Omega}} : f \in \mathcal{N}_L \} \subseteq \{\, u : {\Omega} \to \mathbb{R} \},
\end{equation*}
where $f|_{{\Omega}}$ denotes the restriction of $f$ to $\Omega$. No restriction is imposed on the hidden-layer widths; in particular, networks of different architectures are all included in \(\mathcal{H}\).
For fixed $M\in\N$, we denote by \(\mathcal{H}^M \subset \mathcal{H}\) the subclass of functions realizable by depth-\(L\) neural networks with exactly \(M\) parameters.
Clearly, \(\mathcal{H}\) can be expressed as the union of \(\mathcal{H}^M\) over all admissible values of \(M\).
All weights and biases are collected into a parameter vector \(\theta\in\mathbb{R}^M\), and we write \(f_\theta\) whenever it is convenient to emphasize this dependence.
In the width-unlimited setting, $\mathcal H$ is closed under finite linear combinations: if $u_1,\dots,u_k\in\mathcal H$ and $\lambda_1,\dots,\lambda_k\in\R$, then $\sum_{j=1}^k \lambda_j u_j\in\mathcal H$ for any $k\in\N$, cf.\ \cite[Section~A.1]{Langer:26}.
This property generally fails for a class $\mathcal H^{M}$ with fixed $M\in\N$.

The evaluation of differential operators acting on neural network outputs requires sufficient regularity.
This is automatically satisfied when the activation function \(\sigma\) is $C^\infty$, in which case all derivatives exist in the classical sense.
Common examples include the sigmoid, hyperbolic tangent, softplus, and other \(C^\infty\) activations.

However, we will often employ the rectified linear unit (ReLU), defined by \(\sigma(x)=\max\{0,x\}\) for all $x\in\R$.
Neural networks using this activation are referred to as ReLU neural networks.
Although such functions do not belong to \(C^r\) for \(r\geq 1\), the loss functions considered here only involve pointwise evaluations of \(f_\theta\) and its derivatives at finitely many collocation points.
For ReLU neural networks and $r\in\N_0$, define the $Z^h$-smooth hypothesis class by
\[
\mcHreg{r}(Z^h) :=\left\{ f_\theta \in\mcH\colon f_\theta \text{ is } C^r \text{ at every } z\in Z^h\right\} \subset \mathcal H.
\]
On this subset, all discrete loss terms are well-defined.
This restriction explains why ReLU activations can still be employed in PDE-based learning frameworks despite their limited global regularity.

Analogously, we define \(\mcHreg{r}^M(Z^h) \subset \mcHreg{r}(Z^h)\) as the subset consisting of depth-\(L\) neural networks with exactly \(M\) parameters.

\subsection{A Finite-Measurement Tool}\label{Sec:FiniteLoss}

In order to avoid repetition in the proofs of the non-uniqueness results below, we isolate a technical observation concerning loss functionals that depend only on finitely many measurements of the trial function and its derivatives. Such loss functionals arise naturally in all neural network-based discretizations considered in this work once numerical quadrature or collocation is employed. Although the precise analytical form of the loss varies between methods, its information content is always finite. The purpose of this subsection is to formalize this observation and to identify a generic non-uniqueness mechanism that will be invoked repeatedly in the subsequent proofs.

\subsubsection{Finite-measurement maps}
Let $H$ be a set of functions and let $m\in\mathbb N$. We call a map \(\mathcal M:H\to\mathbb R^m\) a finite-measurement map. A functional \(\mcJ^h:H\to\mathbb R\) is called a finite-measurement loss if there exists a function \(G:\mathbb R^m\to\mathbb R\) such that
\begin{equation}\label{eq:finite_measurement_loss}
\mcJ^h(u) = G\bigl(\mathcal M(u)\bigr)
\end{equation}
for all \(u\in H\).

This abstract formulation covers discrete loss functionals considered in the sequel, including Deep Ritz losses, variational regularization losses with pointwise derivatives, and finite-difference regularizers. 

\subsubsection{A generic non-uniqueness mechanism}

The following lemma identifies a structural source of non-uniqueness that arises whenever the loss
depends only on finitely many measurements.

\begin{lemma}[Non-uniqueness for finite-measurement losses]
\label{lem:finite_measurement_nonuniqueness}
Let $H$ be a set of functions that is closed under finite linear combinations. Let $\mcJ^h : H \to \mathbb{R}$ be a finite-measurement loss of the form \eqref{eq:finite_measurement_loss}, where the finite-measurement map $\mathcal{M}$ is linear.
Assume that there exists a nonzero function $\Phi \in H$ such that
\begin{equation}\label{eq:null_measurement}
\mathcal{M}(\Phi) = 0 .
\end{equation}
Then, for all $u \in H$ and all $\lambda \in \mathbb{R}$,
\begin{equation*}
\mcJ^h(u + \lambda \Phi) = \mcJ^h(u).
\end{equation*}
In particular, if $u$ is a minimizer of $\mcJ^h$ on $H$, then $u + \lambda \Phi$ is also a minimizer for all $\lambda \in \mathbb{R}$, and the minimizer is not unique.
\end{lemma}

\begin{proof}
By linearity of $\mathcal{M}$, we have
\[
\mathcal{M}(u + \lambda \Phi) = \mathcal{M}(u) + \lambda \mathcal{M}(\Phi) = \mathcal{M}(u),
\]
where the last equality follows from \eqref{eq:null_measurement}. Consequently,
\[
\mcJ^h(u + \lambda \Phi) = G\bigl(\mathcal{M}(u + \lambda \Phi)\bigr) = G\bigl(\mathcal{M}(u)\bigr) = \mcJ^h(u).
\]
Since $H$ is closed under finite linear combinations, $u + \lambda \Phi \in H$ for all
$\lambda \in \mathbb{R}$, which concludes the proof.
\end{proof}

\Cref{lem:finite_measurement_nonuniqueness} shows that non-uniqueness is a structural consequence of finite-measurement losses whenever the measurement map $\mathcal{M}$ has a nontrivial nullspace on $H$, independent of the well-posedness of the underlying continuous problem.
In the subsequent sections, we show that for the neural network classes considered here, nontrivial functions $\Phi$ satisfying \eqref{eq:null_measurement} can always be constructed, thereby leading to ill-posed discrete optimization problems.

\section{The Deep Ritz Method}\label{Sec:DRM}
We consider energies admitting a local integral representation
\[
\mcE(u) = \int_\Omega \mcL(\nabla u(z), u(z), z) \dz,
\]
where $\mcL : \mathbb{R}^d \times \mathbb{R} \times \Omega \to \mathbb{R}$ is a local integrand. The associated variational problem reads
\begin{equation}\label{Eq:Energy}
\min_{u \in \mathcal A} \mathcal E(u), \qquad \mathcal A := \bigl\{ u \in  W^{1,p}(\Omega) \colon \mathcal B(\Tr(u)) = 0 \text{ on } \Gamma \bigr\},
\end{equation}
where $\Gamma \subset \partial\Omega$, $\mathcal B$ encodes essential (Dirichlet-type) boundary conditions of order $0$ in the sense of traces on $\Gamma$, and $p\in [1,\infty)$.

The Deep Ritz method~\cite{EYu:18} aims to solve \eqref{Eq:Energy} by replacing $\mcA$ by a set of neural networks. While certain boundary conditions, such as homogeneous Dirichlet boundary conditions, can be imposed exactly by construction in neural network ansatz spaces, essential boundary conditions that are not built into the ansatz are typically enforced weakly in Deep Ritz-type methods.
Since we consider a quite general framework, to simplify our presentation we enforce essential boundary conditions weakly by a penalty approach leading to 
\begin{equation}\label{Eq:DeepRitz}
\min_{u_\theta \in \mcH} \mcE(u_\theta) + \alpha_\beta \int_\Gamma |\mcB(\Tr(u_\theta)(z))|_{}^p\, \ds,
\end{equation}
where $\alpha_\beta \geq 0$ is a penalty parameter. Here, to make \eqref{Eq:DeepRitz} well-defined, we assume $\mcH\subset W^{1,p}(\Omega)$, ensuring that weak gradients and traces on $\Gamma$ exist.

\subsection{Existence of minimizers}
Assume that $\mcA$ is nonempty, that $\mcE$ is proper and coercive on $\mcA$, and that $\mcE$ is weakly lower semicontinuous and bounded from below. Then the direct method of the calculus of variations guarantees the existence of a minimizer of \eqref{Eq:Energy}.
However, note that even if \eqref{Eq:Energy} has a solution, \eqref{Eq:DeepRitz} may fail to have a minimizer because $\mcH$ is generally not compact in $W^{1,p}(\Omega)$ and minimizing sequences may not admit convergent subsequences within $\mcH$. 
However, this issue can be solved by regularizing $\theta$ in \eqref{Eq:DeepRitz} and restricting the solution space to neural networks with a fixed finite number of weights and biases, yielding
\begin{equation}\label{Eq:DeepRitzReg}
\min_{u_\theta \in \mcH^M} \mcE(u_\theta) + \alpha_\beta \int_\Gamma |\mcB(\Tr(u_\theta)(z))|^p\, \ds + \alpha_\theta|\theta|_q = \min_{\theta \in \R^M} \mcE(u_\theta)+ \alpha_\beta \int_\Gamma |\mcB(\Tr(u_\theta)(z))|^p\, \ds+ \alpha_\theta|\theta|_q,
\end{equation}
where $\alpha_\theta > 0$ and $q\in\N\cup \{\infty\}$. 
Assume that the unregularized parameter-level functional
\[
\theta\mapsto E(u_\theta) +\alpha_\beta\int_\Gamma |B(\operatorname{Tr}u_\theta)(x)|^p\,d\Gamma
\]
is lower semicontinuous on \(\mathbb R^M\) and bounded from below.
Then the addition of the coercive parameter penalty \(\alpha_\theta |\theta|_q\), yields existence of a minimizer of \eqref{Eq:DeepRitzReg}.

In order to implement a solution method for \eqref{Eq:DeepRitz} and \eqref{Eq:DeepRitzReg}, the involved integrals need to be approximated by quadrature. Since the resulting discrete loss depends on pointwise values and gradients at the interior collocation points \(\Omega^h\) and on boundary point values at \(\Gamma^h\), it is convenient to work with neural network representatives on \(\overline\Omega\). We therefore pass from the variational trial class on \(\Omega\) to the corresponding class of neural network restrictions on \(\overline\Omega\). In particular, we define
\[
\overline{\mcH} := \{\, f|_{\overline\Omega} : f \in \mathcal N_L \,\} \subseteq \{\, u : \overline\Omega \to \mathbb R \,\}
\]
and
\[
\overmcHreg{r}(Z^h) := \left\{ f_\theta \in \overline{\mathcal H} \;\middle|\; f_\theta \text{ is } C^r \text{ at every } z\in Z^h \right\}.
\]
Further, we define \(\overline\mcH^M \subset \overline{\mcH}\) and \(\overmcHreg{r}^M(Z^h) \subset \overmcHreg{r}(Z^h)\) as the subsets consisting of depth-\(L\) neural networks with exactly \(M\) parameters. For functions in \(\overmcHreg{r}(Z^h)\), the boundary point values \(f_\theta(z)\) for \(z\in\Gamma^h\) are classically well defined, since these functions are defined on \(\overline\Omega\) and are continuous at every point of \(Z^h\).

With these definitions, the corresponding discrete Deep Ritz problems read
\begin{align}
&\min_{u_\theta \in \overmcHreg{1}(Z^h)}\left\{\mcJ_{\DR}^h(u_\theta):=\sum_{z\in\Omega^h} \omega_\mcL^z \mcL(\nabla u_\theta(z), u_\theta(z), z)+\alpha_\mcB \sum_{z\in \Gamma^h} \omega_\mcB^z |\mcB(\Tr(u_\theta)(z))|^p\right\},\label{Eq:DeepRitzDiscrete}\\
&\min_{u_\theta \in \overmcHreg{1}^M(Z^h)}\left\{\mcJ_{\DR}^h(u_\theta) + \alpha_\theta|\theta|_q\right\},\label{Eq:DeepRitzRegDiscrete}
\end{align}
where \(\omega_\mcL^z\) and \(\omega_\mcB^z\) are suitable quadrature weights, for instance \(\omega_\mcL^z = \frac{1}{|\Omega^h|}\) and \(\omega_\mcB^z = \frac{1}{|\Gamma^h|}\). Here, for \(z\in\Gamma^h\), the term \(\Tr(u_\theta)(z)\) is understood as the classical boundary point value \(u_\theta(z)\).

Note that even if the energy in \eqref{Eq:DeepRitz} is coercive its discrete counterpart does not need to be coercive and hence may fail to have a solution. We illustrate that by a simple example. 
\begin{example}\label{Ex:NoSolution}
Let $\zeta\in L^2(\Omega)\cap C(\overline{\Omega})$, then the Poisson energy
\[
\mcE(u) = \frac{1}{2} \int_\Omega |\nabla u|^2 \dz - \int_\Omega \zeta u \dz
\]
is coercive in $H^1_0(\Omega)$. 

Assume now that there is $z_0\in\Omega^h$ with $\zeta(z_0)\not= 0$ and consider the discrete functional
\[
\mcJ_{\DR}^h(u) =  \sum_{z\in\Omega^h} \omega_\mcL^z \left(\frac{1}{2}|\nabla u(z)|^2 -  \zeta(z) u(z)\right),
\]
where $\omega_\mcL^z > 0$ for all $z\in\Omega^h$ and $\nabla u(z)$ denotes the pointwise gradient of $u$ at the sampled collocation point $z$. Since the gradient term is evaluated only at finitely many points, $\sum_{z\in\Omega^h}\omega_\mcL^z|\nabla u(z)|^2$ does not control the nodal values $u(z)$.
In particular, one can choose admissible functions $u$ such that $|\nabla u(z)|$ is small at all $z\in\Omega^h$ while $u(z)$ has the same sign as $\zeta(z)$ and arbitrarily large magnitude at points where $\zeta(z)\neq 0$. Consequently, the linear term $-\sum_{z\in\Omega^h}\omega_\mcL^z\zeta(z)u(z)$ can dominate and drive $\mcJ_{\DR}^h(u)\to -\infty$. Hence, $\mcJ_{\DR}^h$ is not bounded below and not coercive. Moreover, $\mcJ_{\DR}^h$ does not attain its minimum in $\mcH$.
\end{example}

While \eqref{Eq:DeepRitzDiscrete} does not admit a minimizer in general, the existence of a minimizer for \eqref{Eq:DeepRitzRegDiscrete} follows if $\mcJ_{\DR}^h$ is lower semicontinuous and bounded from below, since the term $\alpha_\theta | \theta |_q$ renders the functional in \eqref{Eq:DeepRitzRegDiscrete} coercive on $\mathbb{R}^M$.

\subsection{Non-uniqueness of minimizers}

We now turn to the non-uniqueness of minimizers of the discrete Deep Ritz loss. The following example shows that, even when \eqref{Eq:DeepRitzDiscrete} admits a minimizer, uniqueness may fail.

\begin{example}
Consider the 1D differential equation
\begin{equation}\label{eq:DRM:counterexample1}
\begin{split}
&u''(z) = 0 \quad \text{for } z\in(0,T)\subset \R,\\
&u(0) = u_0\in\R, \qquad u(T) = u_T\in\R.
\end{split}
\end{equation}
The exact solution of \eqref{eq:DRM:counterexample1} is given by $u(z) = \frac{u_T-u_0}{T} z + u_0$. 

Choosing a set of collocation points $Z^h = \{z_i\}_{i=0}^{N+1}$, where $z_0=0$ and $z_{N+1}=T$, and $\overmcHreg{r}(Z^h)$ with $r=1$ as a set of ReLU neural networks, the Deep Ritz method solves
\begin{equation}\label{eq:DRM:counterexample2}
\min_{u_\theta\in \overmcHreg{1}(Z^h)} \frac{1}{N} \sum_{i=1}^N \frac{1}{2}|u_\theta'(z_i)|_2^2 + \alpha_\mcB \left(|u_\theta(0)-u_0|^2 + |u_\theta(T)-u_T|^2\right).
\end{equation}
Since the exact solution of \eqref{eq:DRM:counterexample1} is affine, it can be represented exactly by a ReLU neural network. 
In particular, the discrete loss \eqref{eq:DRM:counterexample2} admits global minimizers with zero loss, and these minimizers are not unique.

To see this, fix arbitrary values $\mu_1,\dots,\mu_N\in\R$. Choose pairwise disjoint intervals $I_i:=(z_i-\varepsilon,z_i+\varepsilon)\subset(0,T)$ with $\varepsilon>0$, and define $u_\theta$ such that
\[
u_\theta(0)=u_0,\qquad u_\theta(T)=u_T,\qquad u_\theta(z)=\mu_i\ \text{for all }z\in I_i,\ i=1,\dots,N,
\]
and $u_\theta$ is affine on each connected component of $(0,T)\setminus\bigcup_{i=1}^N I_i$.
Then $u_\theta$ is continuous and piecewise affine, satisfies $u_\theta'(z_i)=0$ for all $i\in\{1,\dots,N\}$ (since it is constant on $I_i$), and matches the boundary values exactly. Hence the functional in \eqref{eq:DRM:counterexample2} evaluated at $u_\theta$ yields zero. 
Since the values $\mu_1,\dots,\mu_N$ are arbitrary, \eqref{eq:DRM:counterexample2} has infinitely many minimizers.

If we replace $\overmcHreg{1}(Z^h)$ by $\overmcHreg{1}^M(Z^h)$ in \eqref{eq:DRM:counterexample2} the non-uniqueness persists, assuming that $M$ is sufficiently large. 
In fact a class of one-hidden-layer neural networks with 1 hidden neuron, i.e., $M=4$, is sufficiently expressive. Then $u_\theta(z) = \frac{u_T -u_0}{T + b} \sigma(z + b) + u_0$ for any $b \in (-T, -z_N)$ solves \eqref{eq:DRM:counterexample2}. In fact $u_\theta(0) = u_0$, $u_\theta(T)=u_T$ and $u_\theta'(z_i) = 0$ for $i\in\{1,\ldots,N\}$. Since this is a solution for any $b\in (-T, -z_N)$, there are infinitely many minimizers.

For $M=2$ (no hidden layer) one calculates that the unique solution of \eqref{eq:DRM:counterexample2} is 
\begin{equation}\label{Eq:DRM:counterexample3}
u_\theta(z) = \frac{\alpha_\mcB T(u_T - u_0)}{\alpha_\mcB T^2+1} z + \frac{2 \alpha_\mcB T^2 u_0 + u_0 + u_T}{2 \alpha_\mcB T^2+2},
\end{equation}
which is not a solution of \eqref{eq:DRM:counterexample1}. 
If the boundary conditions are enforced strongly, i.e.,
\begin{equation*}
\min_{\substack{u_\theta \in \overmcHreg{1}^M(Z^h) \\ u_\theta(0)=u_0,\; u_\theta(T)=u_T}}
\frac{1}{N}\sum_{i=1}^N |u_\theta'(z_i)|^2,
\end{equation*}
 then in the affine case ($M=2$) the  Deep Ritz problem admits a unique minimizer, which coincides with the exact PDE solution. The inconsistency observed above is of course a consequence of weak boundary enforcement combined with the discrete interior energy. In fact the solution \eqref{Eq:DRM:counterexample3} approaches the exact PDE solution for $\alpha_\mcB\to\infty$.
\end{example}
This example illustrates that the Deep Ritz loss may fail to approximate the continuous energy, leading to non-uniqueness or biased minimizers even for trivial elliptic problems. The following theorem shows the non-uniqueness issue for more general problems.
\begin{theorem}\label{Thm:DRM:Nonuniqueness}
Let $\Omega^h\subset\Omega$ and $\Gamma^h\subset\Gamma$ be fixed finite sets such that \( Z^h := \Omega^h\cup\Gamma^h \neq \emptyset \). Let $H\subset\overline \mcH$ be a class of depth-$L$ neural networks satisfying one of the following assumptions:
\begin{enumerate}[(i)]
\item $H=\overmcHreg{1}(Z^h)$ consists of ReLU neural networks with \(L \ge \big\lceil \log_2(d+1)\big\rceil +1.\)

\item $H=\overmcHreg{\ell}(Z^h)$, where \(\ell := 2|\Omega^h| + |\Gamma^h|,\) consists of neural networks with activation $\sigma\in C^{\ell}(\R)$ such that there exists $s\in\R$ with \(D^k\sigma(s)\neq 0\) for all $0\le k\le \ell$, and depth $L\ge 2$. If $L>2$, we additionally assume that $\sigma$ is strictly monotone.
\end{enumerate}
If the minimization problem $\argmin_{u_\theta\in H} \mcJ_{\DR}^h(u_\theta)$ admits a solution, then the set of minimizers is infinite.
\end{theorem}

\begin{proof}
We observe that the discrete Deep Ritz loss \eqref{Eq:DeepRitzDiscrete} is a finite-measurement loss in the sense of \cref{Sec:FiniteLoss}. More precisely, it can be written in the form
\[
\mcJ_{\DR}^h(u_\theta)=G\bigl(\mathcal{M}_{\DR}(u_\theta)\bigr),
\]
where \(\mathcal{M}_{\DR}\) is the linear measurement map collecting the interior values \(u_\theta(z)\), the interior gradients \(\nabla u_\theta(z)\) for \(z\in\Omega^h\), and the boundary trace values \(\operatorname{Tr}(u_\theta)(z)=u_\theta (z)\) for \(z\in\Gamma^h\), i.e., 
\[
\mathcal{M}_{\DR}:\overmcHreg{1}(Z^h) \to \mathbb R^{|\Omega^h|(d+1)+|\Gamma^h|},\qquad \mathcal{M}_{\DR}(u) = \bigl((u(z),\nabla u(z))_{z\in\Omega^h},(\Tr(u)(z))_{z\in\Gamma^h}\bigr).
\]
In case \((i)\), let \(H= \overmcHreg{1}(Z^h)\), and in case \((ii)\), let \(H=\overmcHreg{\ell}(Z^h)\). By \cite[Lemma A.8]{Langer:26} in case \((i)\) and \cite[Lemma A.4]{Langer:26} in case \((ii)\) there exists a nonzero function $\Phi\in H$ such that \( \mathcal M_{\DR}(\Phi)=0.\) The claim therefore follows directly from \cref{lem:finite_measurement_nonuniqueness} by noting that $H\subseteq \overmcHreg{1}(Z^h)$ and that $H$ is closed under linear combinations. 
\end{proof}

\begin{remark}[Hard enforcement at sampled boundary points]
\Cref{Thm:DRM:Nonuniqueness} remains valid if the boundary conditions are imposed as hard constraints only on the finite boundary set \(\Gamma^h\), i.e.,
\begin{equation}\label{Eq:DRM:DiscreteConstrained}
\min_{\substack{u_\theta \in H \\ \mcB(\Tr(u_\theta)(z)) = 0\ \forall z\in \Gamma^h}}\sum_{z \in \Omega_h} \omega_{\mathcal L}^z \,\mathcal L(\nabla u_\theta(z), u_\theta(z), z),
\end{equation}
where $H$ as in \cref{Thm:DRM:Nonuniqueness}. Indeed, the constraints then depend only on finitely many measurements, and the null perturbations used in the proof can be chosen to vanish on these measurements. 
This is different from strong enforcement on the whole boundary \(\Gamma\), which would require additional assumptions ensuring that the null perturbations satisfy the corresponding homogeneous boundary condition on all of \(\Gamma\).
\end{remark}

\begin{remark}[Distance to the true solution]\label{Rem:DistanceSolution}
In both the weak and hard enforcement cases, the discrete problem admits infinitely many minimizers, and infinitely many of them may be arbitrarily far from the solution of the continuous problem.
Let $H$ be one of the neural trial classes from \Cref{Thm:DRM:Nonuniqueness}, let $u^* \in W^{1,p}(\Omega)$ denote the solution of the underlying variational problem \eqref{Eq:Energy}, let $\hat u_\theta\in H$ be a minimizer of the discrete Deep Ritz loss \eqref{Eq:DeepRitzDiscrete} or \eqref{Eq:DRM:DiscreteConstrained}, and let $\Phi \in H$ be as in the proof of \Cref{Thm:DRM:Nonuniqueness}.
Since $\Phi \not\equiv 0$ on $\Omega$, we have $\|\Phi\|_{L^p(\Omega)} > 0$, and therefore
\[
\|\hat u_\theta + \lambda \Phi - u^*\|_{L^p(\Omega)} \ge |\lambda| \|\Phi\|_{L^p(\Omega)} - \|\hat u_\theta - u^*\|_{L^p(\Omega)} \;\longrightarrow\; \infty \quad \text{as } |\lambda| \to \infty.
\]
Moreover, the degeneracy is stronger than the existence of a single unbounded affine family of minimizers. In the proof of \Cref{Thm:DRM:Nonuniqueness}, the null function $\Phi$ is constructed by choosing a point $z_0 \in \Omega \setminus Z^h $ at which $\Phi$ is nonzero, while all measurements entering the discrete loss vanish, see \cite[Lemma A.4 \& Lemma A.8]{Langer:26}. 
Since $\Omega$ is open and $Z^h $ is finite, there are infinitely many such choices of $z_0$. Consequently, there exist infinitely many distinct nonzero null functions $\Phi \in H$, and hence infinitely many distinct affine families of minimizers of the form $\hat u_\theta + \lambda \Phi$, $\lambda \in \R$. 
In particular, for every sufficiently large $\epsilon>0$, for instance for $\epsilon>\|\hat u_\theta - u^*\|_{L^p(\Omega)}$, there exist infinitely many discrete minimizers $\tilde u_\theta \in H$ of $\mcJ_{\DR}^h$ such that
\[
\|\tilde u_\theta - u^*\|_{L^p(\Omega)} = \epsilon.
\]
\end{remark}

\begin{remark}[Width-limited architectures]
In \Cref{Thm:DRM:Nonuniqueness}, the hypothesis class $H$ is taken to be width-unlimited in order to guarantee closure under linear combinations. Let $\hat{u}_\theta\in H$ be a minimizer of \eqref{Eq:DeepRitzDiscrete} or \eqref{Eq:DRM:DiscreteConstrained} and $\Phi\in H$ as in the proof of \cref{Thm:DRM:Nonuniqueness}. 
Then, however, the constructed family of minimizers $\{\hat u_\theta + \lambda \Phi \colon \lambda\in\R\}$ consists of neural networks of finite width.
Consequently, the same non-uniqueness phenomenon persists for width-limited neural network classes, provided that the architecture is sufficiently expressive to realize the linear combination $\hat u_\theta + \lambda \Phi$.

We emphasize that this observation is qualitative. In contrast to the strong-form PINN setting in \cite{Langer:26}, we do not establish here a finite-width realizability result for discrete Deep Ritz minimizers analogous to \cite[Proposition 3.1]{Langer:26}, and therefore do not obtain an explicit architecture-level non-uniqueness statement. The reason is that such a result would require interpolation of the finitely many quantities entering the discrete loss, in particular pointwise gradient data, whereas the present variational setting is naturally formulated in terms of weak derivatives. To the best of our knowledge, a corresponding neural-network Hermite interpolation result is not available.
\end{remark}

\begin{remark}[Ill-posedness beyond minimizers]\label{Rem:IllPosedBeyondMinimizers}
In \cref{Thm:DRM:Nonuniqueness}, the result is formulated under the assumption that a minimizer of the discrete loss functional exists. This assumption is made for clarity but is not essential. Indeed, the structural degeneracy underlying the result (see \cref{Rem:DistanceSolution}) is independent of the existence of a minimizer and follows directly from the finite-measurement structure characterized in \cref{lem:finite_measurement_nonuniqueness}.

More precisely, let $H$ be as in \cref{Thm:DRM:Nonuniqueness}, then for any admissible trial function $u_\theta \in H$ there exist infinitely many distinct trial functions $\tilde u_\theta \in H$ with $\tilde u_\theta\neq u_\theta$ attaining the same loss value, i.e., $\mcJ^h_{\DR}(u_\theta) = \mcJ^h_{\DR}(\tilde u_\theta)$. This is a consequence of the nontrivial nullspace induced by the finite number of linear measurements defining the loss functional.
As a result, even in the absence of a minimizer, or when an optimization algorithm is terminated at an attained loss value, the discrete problem admits infinitely many candidates that are indistinguishable by the loss.

This observation applies uniformly to all discretizations falling within the finite-measurement framework described in \cref{Sec:FiniteLoss}.
\end{remark}

\begin{remark}\label{Rem:Omega}
Although we define $\overmcHreg{r}(Z^h)$ as a class of functions restricted to $\overline\Omega$, the statement of \cref{Thm:DRM:Nonuniqueness} remains valid if the neural networks are considered as globally defined functions $f_\theta:\mathbb R^d\to\mathbb R$. 
Indeed, the loss $\mcJ_{\DR}^h$ and the associated measurement map only depend on the values of $f_\theta$ and its derivatives on the finite set $Z^h$, and are therefore insensitive to modifications of $f_\theta$ outside $\overline\Omega$.
\end{remark}

While the discrete problem is generally ill-posed and admits non-unique minimizers, it is nevertheless possible to characterize conditions under which the minimizing solution is uniquely determined on the collocation points.

\begin{proposition}[Uniqueness on collocation points]\label{Prop:DRM:uniqueness}
Let \(H \subset \overmcHreg{1}(Z^h)\) be a neural trial class that is closed under convex combinations, and let \(\Omega^h\subset \Omega\) and \(\Gamma^h\subset \Gamma\) be fixed finite sets such that \(Z^h:=\Omega^h \cup \Gamma^h \neq \emptyset\). 
Suppose that \(\omega_{\mcL}^z>0\) for all \(z\in\Omega^h\), \(\omega_{\mcB}^z>0\) for all \(z\in\Gamma^h\), and \(\alpha_{\mcB}\ge 0\).
Assume that, for every \(z\in\Omega^h\), the map
\[
(\xi,s)\mapsto \mcL(\xi,s,z)
\]
is strictly convex. Moreover, if $\alpha_\mcB>0$, the boundary penalty is strictly convex in the trace values (e.g., $p>1$ and $\mcB$ is linear and injective).
If \eqref{Eq:DeepRitzDiscrete} possesses a solution, then any two minimizers $u_{\theta_1},u_{\theta_2}\in H$ of \eqref{Eq:DeepRitzDiscrete} satisfy
\[
\nabla u_{\theta_1}(z)=\nabla u_{\theta_2}(z), \quad u_{\theta_1}(z)=u_{\theta_2}(z) \qquad \forall z\in\Omega^h
\]
and if $\alpha_\mcB > 0$,
\[
\Tr(u_{\theta_1})(z)=\Tr(u_{\theta_2})(z)\quad \forall z\in\Gamma^h.
\]
\end{proposition}
\begin{proof}
Assume that \(u_{\theta_1},u_{\theta_2}\in H\) are two minimizers and let
\[
u:=\lambda u_{\theta_1}+(1-\lambda)u_{\theta_2},\qquad \lambda\in(0,1).
\]
Since \(H\) is closed under convex combinations, we have \(u\in H\).
If, for some \(z\in\Omega^h\),
\[
(\nabla u_{\theta_1}(z),u_{\theta_1}(z)) \neq (\nabla u_{\theta_2}(z),u_{\theta_2}(z)),
\]
then strict convexity of \(\mcL\) in its first two arguments, together with positivity of the quadrature weights, gives
\[
\mcJ^h_{\DR}(u) < \lambda \mcJ^h_{\DR}(u_{\theta_1}) + (1-\lambda)\mcJ^h_{\DR}(u_{\theta_2}),
\]
contradicting minimality. 

If \(\alpha_{\mathcal B}>0\) and \(\Tr(u_{\theta_1})(z)\neq \Tr(u_{\theta_2})(z)\) for some \(z\in\Gamma^h\), then the linearity of the trace and the strict convexity of the boundary penalty in the trace values yield the same contradiction.
\end{proof}
Note that the objective in \eqref{eq:DRM:counterexample2} is not strictly convex in $u_\theta$ and hence distinct minimizers do not necessarily coincide on collocation points. More generally we remark the following:

\begin{remark}[Derivative-only energies do not control point values]
Consider a Deep Ritz discretization whose integrand is independent of the function value, i.e., $\mathcal{L}(\nabla u,u,z)=\mathcal{L}(\nabla u,z)$ (with a slight abuse of notation), and impose Dirichlet data only through a boundary penalty or hard constraint. Then the resulting discrete loss functional does not control interior point values of $u$.

Indeed, for any $u_\theta\in {H} \subseteq \overmcHreg{1}(Z^h)$ satisfying the imposed boundary condition and any finite set of interior points $\Omega^h$, there exists $\tilde u_\theta\in {H}$ satisfying the same boundary condition such that
\[
\nabla \tilde u_\theta(z)=\nabla u_\theta(z)\quad\text{for all }z\in\Omega^h,
\]
but $\tilde u_\theta(z)\neq u_\theta(z)$ for at least one $z\in\Omega^h$. Consequently, the discrete energy (and hence the loss) cannot enforce agreement of trial functions at the collocation points.
\end{remark}

Hence, the Deep Ritz method shares the same fundamental limitation as PINNs: even if the underlying differential equation admits a unique solution, the corresponding discretized optimization problem \eqref{Eq:DeepRitzDiscrete} may admit infinitely many minimizers, or fail to be solvable, thereby rendering the discrete optimization problem ill-posed; cf.\ \cite{Langer:26} for the analogous result for PINNs.

Note that \cref{Thm:DRM:Nonuniqueness,Prop:DRM:uniqueness} do not apply to \eqref{Eq:DeepRitzRegDiscrete}. Establishing an analogous result in this setting is more delicate, since the penalty term $\alpha_\theta | \theta|_q$ acts directly on the weights and biases of the minimizers. It is therefore unclear whether the penalization restores uniqueness. A rigorous investigation of this question, likely requiring different analytical techniques, is beyond the scope of the present work.

The ill-posedness identified above is a consequence of replacing the continuous energy by a finite quadrature formula. Possible strategies for dealing with the resulting effects are discussed in \cite{RiveraTaylorOmellaPardo2022}, including Monte Carlo integration, adaptive integration, polynomial approximation, and regularization. These strategies are treated there mainly from a practical and numerical perspective. For the regularization approach, an additional analytical motivation is provided only in a simplified setting, and its extension to more general network architectures or nonsmooth integrands is not clear. Whether these strategies merely improve practical behavior or can also remove the ill-posedness identified here is beyond the scope of the present work.

\section{Regularization Methods}\label{Sec:Regularization}

For a fixed integer $m\geq 1$, we denote by \(\big(D^\beta u(x)\big)_{1\leq |\beta|\leq m}\in\mathbb{R}^{\eta}\) the collection of all weak partial derivatives of $u: \Omega \to \R$ at $x\in\Omega$ of total order between $1$ and $m$, where \(\eta \;=\; \sum_{k=1}^m \binom{d+k-1}{k}.\)
Let $R\colon\mathbb{R}^{\eta}\to[0,\infty]$ be a Borel-measurable function with $R(0)=0$. We then define the regularization functional
\begin{equation}\label{Eq:Regularizer}
\mathcal{R}(u) \;:=\; \int_{\Omega} R \left(\left(D^\beta u(x)\right)_{1\leq |\beta|\leq m} \right )\dx,
\end{equation}
with the understanding that $\mcR(u)=+\infty$ whenever $R\left(\left(D^\beta u(x)\right)_{1\leq |\beta|\leq m} \right )\notin L^1(\Omega)$.

Given data $g\in L^2(\Omega)$ and weights $\alpha_1,\alpha_2\ge 0$, we then consider the variational problem
\begin{equation}\label{Eq:Regularization:ContProblem}
\min_{u\in \mcA} \left\{ \mcJ_{\reg}(u) := \alpha_1\|u-g\|_{L^1(\Omega)} + \alpha_2\|u-g\|_{L^2(\Omega)}^2 + \mathcal{R}(u)\right\},
\end{equation}
where $\mcA$ is a suitable admissible function space (e.g., a weakly closed subset of $W^{m,p}(\Omega)$ for some $p\ge 1$).
In order to guarantee the existence of minimizers of \eqref{Eq:Regularization:ContProblem}, additional assumptions on $R$ ensuring coercivity and weak lower semicontinuity of $\mathcal R$ are typically required.
The formulation of $\mcR$ is deliberately kept general in order to cover a broad class of variational regularizers commonly used in imaging and inverse problems.
In particular, the considered framework encompasses first- and higher-order, convex and non-convex regularization functionals commonly used in practice, including:
\begin{enumerate}
\item \emph{First-order convex regularization (Tikhonov):}
$\mcR(u) = \|\nabla u\|_{L^p(\Omega)}^p$ for $p \in \N$.

\item \emph{First-order nonsmooth regularization (total variation):}
$\mcR(u) = \int_\Omega |\nabla u|_\nu \dx$, which corresponds to the anisotropic total variation (TV) for $\nu = 1$ and the isotropic TV for $\nu = 2$
\cite{RuOsFa:92}.

\item \emph{Second-order regularization (Hessian-based):}
$\mcR(u) = \int_\Omega |\nabla^2 u|_2 \dx$ \cite{Scherzer1998, LysLunTai:03, HinterbergerScherzer:06, BergouniouxPiffet2010, LaTaCh:13}.

\item \emph{Mixed first- and second-order regularization (TV--Hessian):}
\[
\mcR(u) = \int_\Omega \rho(x)\,|\nabla^2 u|_2 + \bigl(1-\rho(x)\bigr)\,|\nabla u|_2 \dx,
\]
where $\rho\colon\Omega\to[0,1]$ is a locally adaptive weight \cite{LysakerTai:06, PapafitsorosSchonlieb:14}.

\item \emph{Edge-adaptive higher-order regularization (TV--Laplacian):}
\[
\mcR(u) = \rho_1 \int_\Omega |\nabla^2 u|_2 \dx
+\rho_2 \int_\Omega \rho(|\nabla u|_2) \,(\Delta u)^2 \dx,
\] 
where $\rho_1,\rho_2 \ge 0$, and $\rho:[0,\infty) \to [0,\infty)$ is a smooth, nonincreasing edge-adaptive weight that suppresses higher-order smoothing near edges \cite{ChMaMu:00}.

\item \emph{Geometric regularization (Euler's elastica):}
$\mcR(u) = \int_\Omega \left( \rho_1 + \rho_2 \left(\nabla \cdot \frac{\nabla u}{|\nabla u|} \right)^2 \right) |\nabla u| \dx$, with parameters $\rho_1,\rho_2 \ge 0$ \cite{ChKaSh:02, TaHaCh:11}.

\item \emph{Non-convex regularization:}
$\mcR(u) = \int_\Omega |\nabla u|_2^p \dx$ for $0 < p < 1$ \cite{Chartrand2007, NiNgZhCh:08}.
\end{enumerate}

The restriction $1 \le |\beta|$ excludes zeroth-order regularizers of the form $\|u\|_{L^p(\Omega)}$ for $p>0$, which do not control derivatives and therefore fall outside the class of variational models considered here.

The choice of the data fidelity term is typically closely tied to the statistical properties of the noise affecting the observations. A purely quadratic $L^2$ data term corresponds to the classical least-squares formulation and is optimal under additive Gaussian noise assumptions; see for example \cite{ChCaCrNoPo:10}.
In contrast, an $L^1$ data fidelity term is well-known to be robust with respect to impulsive noise and has been studied in the context of total variation regularization; see, e.g., \cite{Alliney:97, Nikolova:02, Nikolova:04}.
In practical applications, data are often corrupted by a combination of Gaussian noise and impulsive perturbations. To account for such mixed noise models within a single variational formulation, combined $L^1$-$L^2$ data fidelity terms have been proposed, leading to minimization problems of the form considered here; see, in particular, \cite{HintermullerLanger:13, Langer:17a}. The mixed $L^1$-$L^2$ fidelity thus provides a flexible data term that interpolates between the two classical noise models.

Replacing $\mcA$ in \eqref{Eq:Regularization:ContProblem} by a set of neural networks $\mcH \subseteq \mcA$ yields
\begin{equation}\label{Eq:Regularization:NN}
\argmin_{u_\theta\in\mcH} \mcJ_{\reg}(u_\theta).
\end{equation}
Such an approach has been recently proposed in \cite{LangerBehnamian:24} in the context of total variation minimization. In particular, there, a constrained version of \eqref{Eq:Regularization:NN} is considered, i.e., instead of $\mcH$ the functional in \eqref{Eq:Regularization:NN} is optimized over the compact set $\mcH^M_c:=\{u_\theta\in\mcH^M \colon |\theta|_\infty \le c\}$. This, in fact, guarantees the existence of minimizers of $\mcJ_{\reg}$, if $\mcR$ is lower semicontinuous, which is the case for total variation; cf.\ \cite[Theorem 3.4]{LangerBehnamian:24}. 

To practically evaluate the integrals in \eqref{Eq:Regularization:NN} numerical quadrature is usually required. Since the resulting discrete loss depends only on pointwise values of \(u_\theta\) and its derivatives at the interior quadrature points \(\Omega^h\), it is natural to work on the pointwise regular neural trial classes introduced in \cref{sec:notation-prelim}, i.e., \(\mcHreg{m}(\Omega^h)\) and \(\mcHreg{m}^M(\Omega^h)\). With this notation, the corresponding discrete regularization problem reads
\begin{equation}\label{Eq:Regularization:QuadratureLoss}
\argmin_{u_\theta\in\mcHreg{m}(\Omega^h)} \left\{ \mcJ_{\reg}^h(u_\theta):= \sum_{\tau=1}^2 \alpha_\tau \sum_{x\in\Omega^h} \omega_\mcD^x |u_\theta(x) - g^h(x)|^\tau + \mcR^h(u_\theta) \right\},
\end{equation}
where $\mcR^h(u_\theta) = \sum_{x\in\Omega^h} \omega_\mcR^x R \left(\left(D^\beta u_\theta(x)\right)_{1\leq |\beta|\leq m} \right )$, $g^h:\Omega^h\to \R$ denotes the restriction (or sampling) of $g$ to the grid $\Omega^h$, and $\omega_\mcD^x, \omega_{\mcR}^x\geq 0$ are respective quadrature weights typically chosen as $\omega_\mcD^x = \omega_{\mcR}^x = \frac{1}{|\Omega^h|}$.
Note that only the integrals are approximated by quadrature in \eqref{Eq:Regularization:QuadratureLoss} and the differential operators $D^\beta$ are  evaluated pointwise.
 
We now analyze the variational problem associated with $\mcJ_{\reg}^h$ and establish existence of minimizers as well as non-uniqueness.
\begin{theorem}[Existence and non-uniqueness]\label{Thm:Regularization:Existence:NonUniqueness}
Let $m\in\N$, let $\Omega^h\subset\Omega$ be a nonempty finite set, and let $H \subseteq \mcH$ be a class of depth-$L$ neural networks satisfying one of the following:
\begin{enumerate}[(i)]
\item $H = \mcHreg{m}(\Omega^h)$ consists of ReLU neural networks of depth $L \ge \lceil \log_2(d+1)\rceil +1$.

\item $H= \mcHreg{\ell}(\Omega^h)$, where \(\ell := |\Omega^h|(m+1),\) consists of neural networks with smooth activation functions $\sigma\in C^{\ell}(\R)$ such that there exists $s\in\R$ with $D^k \sigma(s)\not=0$ for all $0\le k \le \ell$ and depth $L\ge 2$.
If $L>2$, we additionally assume that $\sigma$ is strictly monotone.
\end{enumerate}
Then $\argmin_{u_\theta\in H} \mcJ_{\reg}^h(u_\theta)$ always has a solution $u_\theta \in H$ with $\mcJ_{\reg}^h(u_\theta) = 0$. In particular, there exist infinitely many such solutions. 
\end{theorem}
\begin{proof}
The functional $\mcJ_{\reg}^h$ depends only on finitely many pointwise values of $u_\theta$ and its derivatives up to order $m$ at the quadrature points $\Omega^h$. Hence, $\mcJ_{\reg}^h$ can be written in the form \(\mcJ_{\reg}^h(u_\theta) =G\left(\mathcal M_{\reg}(u_\theta)\right),\) where the linear measurement map $\mathcal M_{\reg}: \mcHreg{m}(\Omega^h)\to \R^{|\Omega^h|(\eta+1)}$ is given by
\[
\mathcal M_{\reg}(u_\theta) = \bigl(u_\theta(x), \ (D^\beta u_\theta(x))_{1 \le |\beta| \le m} \bigr)_{x \in \Omega^h}.
\]
Note that $H \subseteq \mcHreg{m}(\Omega^h)$, since $\ell> m$, where $H=\mcHreg{m}(\Omega^h)$ in case \((i)\) and $H=\mcHreg{\ell}(\Omega^h)$ in case \((ii)\). By the finite-measurement representation above, $\mcJ_{\reg}^h$ is a finite-measurement loss in the sense of \cref{Sec:FiniteLoss}. 

\emph{Existence.}
We construct $u_\theta \in H$ such that
\begin{equation}
\label{Eq:InterpConds}
u_\theta(x)=g^h(x)\quad\text{for all }x\in\Omega^h, \qquad D^\beta u_\theta(x)=0\quad\text{for all }x\in\Omega^h,\ 1\le|\beta|\le m.
\end{equation}
Equivalently, $\mathcal M_{\reg}(u_\theta)=y$ where $y$ encodes the conditions \eqref{Eq:InterpConds}. For ReLU neural networks, such a $u_\theta$ exists by the interpolation construction of \cref{Lem:ReLU:InterpolationNN} (proved in \cref{Sec:ReLUInterpolation}), while for smooth activation functions, the existence follows from \cref{Lem:Smooth:InterpolationNN} (proved in \cref{Sec:ReLUInterpolation}). In either case, \eqref{Eq:InterpConds} implies $\mcJ_{\reg}^h(u_\theta)=0$. Since $\mcJ_{\reg}^h\ge 0$, it follows that $u_\theta$ is a minimizer and $\arg\min_{u_\theta\in H} \mcJ_{\reg}^h(u_\theta)\neq\emptyset$.

\emph{Infinitely many minimizers.}
We next construct a nontrivial $\Phi\in H$ such that
\begin{equation*}
\mathcal M_{\reg}(\Phi)=0.
\end{equation*}
For ReLU neural networks, this is provided by \cite[Lemma~A.8]{Langer:26}; for smooth activation functions, by \cite[Lemma~A.4]{Langer:26}. In particular, $\Phi\not\equiv 0$ on $\Omega$ while $D^\beta\Phi(x)=0$ for all $x\in\Omega^h$ and $1\le |\beta|\le m$. Since $H$ is closed under linear combinations, \cref{lem:finite_measurement_nonuniqueness} yields
\[
\mcJ_{\reg}^h(u_\theta+\lambda\Phi)=\mcJ_{\reg}^h(u_\theta)=0 \qquad \text{for all }\lambda\in\mathbb R,
\]
and therefore $u_\theta+\lambda\Phi \in H$ is a minimizer for every $\lambda\in\mathbb R$. Thus, the set of minimizers is infinite.
\end{proof}

\begin{remark}[Distance to true solution]
It is clear that if a solution $u_\theta$ with $\mcJ_{\reg}^h(u_\theta)=0$ exists, then no regularization has been performed, as the regularization term is just vanishing. Moreover, \cref{Thm:Regularization:Existence:NonUniqueness} shows that the optimization problem \eqref{Eq:Regularization:QuadratureLoss} is ill-posed, yielding infinitely many solutions. In fact, a solution can be arbitrarily far away from a solution of the continuous problem \eqref{Eq:Regularization:ContProblem}. Let $\mcA=W^{m,p}(\Omega)$, $u^*\in W^{m,p}(\Omega)$ be a solution of \eqref{Eq:Regularization:ContProblem} and $u_\theta\in\mcHreg{m}(\Omega^h) \subseteq \mcA$ be a solution of \eqref{Eq:Regularization:QuadratureLoss}, then we observe
\[
\|u_\theta + \lambda \Phi - u^*\|_{L^p(\Omega)} \geq |\lambda|\|\Phi\|_{L^p(\Omega)} - \|u_\theta - u^*\|_{L^p(\Omega)} \longrightarrow \infty \qquad \text{for } |\lambda|\to \infty.
\]
\end{remark}

The same existence and non-uniqueness phenomenon stated in \cref{Thm:Regularization:Existence:NonUniqueness} persists for width-limited architectures, provided the network class is sufficiently expressive. In the case of smooth activations and depth \(L=2\), the required network size can be made explicit. By \cref{Thm:Regularization:Existence:NonUniqueness} and \cref{Lem:Smooth:InterpolationNN}, a minimizer of \(\mcJ^h_{\reg}\) can be realized by a depth-2 network \(u_\theta\in\mcHreg{m}^{\tilde M}(\Omega^h)\) with \(\ell=|\Omega^h|(m+1)\) hidden units, corresponding to \(\tilde M=(d+1)\ell+(\ell+1)\) parameters, and satisfying \(\mcJ^h_{\reg}(u_\theta)=0\). The null perturbation \(\Phi\) used in the proof of \cref{Thm:Regularization:Existence:NonUniqueness} admits a depth-2 representation with \(\ell+1\) hidden units, cf.\ \cite[Remark~A.6]{Langer:26}. Hence, \(u_\theta+\lambda\Phi\) can be represented by a depth-2 network with \(d_1=2\ell+1\) hidden units, corresponding to \(M=(d+1)d_1+(d_1+1)\) parameters.
Therefore, the non-uniqueness asserted in \cref{Thm:Regularization:Existence:NonUniqueness} already occurs within \(\mcHreg{m}^M(\Omega^h)\) for \(M\ge (d+1)(2\ell+1)+(2\ell+2)\).

Moreover, as in the Deep Ritz setting, the ill-posedness identified here is not tied to the existence of minimizers and the degeneracy extends to all loss sublevel sets, cf.\ \cref{Rem:IllPosedBeyondMinimizers}.

As already illustrated in \cite{LangerBehnamian:24} by an example for total variation minimization, and as further substantiated by \cref{Thm:Regularization:Existence:NonUniqueness}, the direct use of pointwise derivatives, for example, obtained via automated differentiation, is problematic for regularization problems of the form \eqref{Eq:Regularization:ContProblem}. 
The underlying reason is that pointwise derivatives are evaluated at individual locations, while the regularization functional $\mcR$ is approximated by numerical quadrature and therefore probed only at finitely many points. 
In general, this prevents a consistent approximation of $\mcR$.

If a neural network discretization is employed, one possible way to address this issue is to discretize the differential operators appearing in \(\mcR\) as well, so that the regularization term is evaluated through finite-dimensional stencil values rather than isolated pointwise derivatives. 
This approach was advocated in \cite{LangerBehnamian:24} for total variation minimization, where finite difference approximations were used.
In a finite difference discretization, the derivative operators $D^\beta$ appearing in $R$ are replaced by finite difference approximations $D_h^\beta$,
for $1\le |\beta|\le m$, leading to \(R\left(\left(D_h^\beta u(x)\right)_{1\leq|\beta|\leq m}\right).\) Accordingly, the finite difference analogue of \eqref{Eq:Regularization:QuadratureLoss} reads as

\begin{equation}\label{Eq:Regularization:FullyDiscrete}
\argmin_{u_\theta\in\mcH} \left\{ \mcJ_{\reg}^{h,\theta}(u_\theta):= \sum_{\tau=1}^2 \alpha_\tau \sum_{x\in\Omega^h} \omega_\mcD^x |u_\theta(x) - g^h(x)|^\tau + \mcR_\theta^h(u_\theta) \right\},
\end{equation}
where $\mcR_\theta^h(u_\theta) = \sum_{x\in\Omega^h} \omega_\mcR^x R \left(\left(D_h^\beta u_\theta(x)\right)_{1\leq |\beta|\leq m} \right )$. 
In contrast to \eqref{Eq:Regularization:QuadratureLoss}, the finite-difference formulation only requires pointwise values of \(u_\theta\) on the grid or stencil, and therefore does not require pointwise differentiability up to order \(m\). Moreover, the existence of solutions of \eqref{Eq:Regularization:FullyDiscrete} is not automatically guaranteed and further standard assumptions, like coercivity, boundedness from below and lower semicontinuity of $\mcJ_{\reg}^{h,\theta}$ are required to guarantee minimizers. Nevertheless, if \eqref{Eq:Regularization:FullyDiscrete} has a solution, then it has infinitely many as the following statement shows.

\begin{theorem}[Non-uniqueness of minimizers]\label{Thm:Regularization:FullyDiscrete:Nonunique}
Consider a nonempty finite collocation set $\Omega^h \subset{\Omega}$, let $L\ge 2$ and let $H \subseteq \mcH$ denote the set of depth-$L$ neural networks with either
\begin{enumerate}[(i)]
\item ReLU activation functions, or
\item non-polynomial activation functions $\sigma\in C(\mathbb R)$ that are strictly monotone if $L>2$, or
\item activation functions $\sigma \in C^{\ell}(\R)$, $\ell:= |\Omega^h|$, satisfying $D^k \sigma(s) \neq 0$ for some $s\in\mathbb{R}$ and all $0 \le k \le \ell$ that are strictly monotone if $L>2$.
\end{enumerate}
If $\argmin_{u_\theta \in H} \mcJ_{\reg}^{h,\theta}(u_\theta)$ has a solution, then it has infinitely many solutions in $\mcH$.
\end{theorem}
\begin{proof}
The functional $\mcJ_{\reg}^{h,\theta}$ depends only on finitely many pointwise values of $u_\theta$ on the grid $\Omega^h$. Hence, $\mcJ_{\reg}^{h,\theta}$ can be written in the form $\mcJ_{\reg}^{h,\theta}(u_\theta) = G(\mathcal M_\theta(u_\theta))$, where the linear map $\mathcal M_\theta: \mcHreg{0}(\Omega^h) \to \R^{|\Omega^h|}$ is given by
\[
\mathcal M_\theta (u_\theta) = (u_\theta(x))_{x\in\Omega^h}.
\]
Note that the ReLU activation is in $C(\R)$ and hence all considered neural networks here are in $\mcHreg{0}(\Omega^h)$.
By the finite-measurement representation above, $\mcJ_{\reg}^{h,\theta}$ is a finite-measurement loss in the sense of \cref{Sec:FiniteLoss}. We construct a $\Phi\in\mcH$ such that $\mathcal M_\theta (\Phi)=0$ and $\Phi\not\equiv 0$ on $\Omega$. Such a depth-2 neural network can be achieved by utilizing \cite[Theorem 5.1]{Pinkus:95} (for (i) and (ii)) and \cite[Lemma A.4]{Langer:26} (for (iii)). To obtain a depth-$L$ neural network $\Phi$ with $L>2$ we add layers that do not change the interpolation conditions and keep a non-zero value in $\Omega\setminus\Omega^h$, such that $\mathcal M_\theta (\Phi)=0$ and $\Phi\not\equiv 0$ on $\Omega$. This can be realized as in the proofs of \cite[Lemma A.4 \& Lemma A.8]{Langer:26}.

Let $\hat{u}_\theta$ be a solution of $\argmin_{u_\theta \in H} \mcJ_{\reg}^{h,\theta}(u_\theta)$, then \cref{lem:finite_measurement_nonuniqueness} yields  $\mcJ_{\reg}^{h,\theta}(\hat{u}_\theta + \lambda \Phi) = \mcJ_{\reg}^{h,\theta}(\hat{u}_\theta)$ for any $\lambda\in\R$, yielding infinitely many solutions of $\argmin_{u_\theta\in H} \mcJ_{\reg}^{h,\theta}(u_\theta)$ in $H$.
\end{proof}

In \cite{LangerBehnamian:24} it is shown that any minimizer $u_\theta \in \mcH$ of \eqref{Eq:Regularization:FullyDiscrete} is closely connected to a minimizer $u^h\in\R^{|\Omega^h|}$ of 
\begin{equation}\label{Eq:Reg:FDM}
\argmin_{u^h\in\R^{|\Omega^h|}} \left\{ \mcJ_{\reg}^{\FD}(u^h):= \sum_{\tau=1}^2 \alpha_\tau \sum_{x\in\Omega^h}|u^h(x) -g^h(x)|^\tau + \mcR_{\FD}^h(u^h) \right\},
\end{equation}
which represents a finite difference discretization of \eqref{Eq:Regularization:ContProblem} where 
\[
\mcR_{\FD}^h(u^h) = \sum_{x\in\Omega^h} \omega_\mcR^x R \left(\left(D_h^\beta u^h(x)\right)_{1\leq |\beta|\leq m} \right )
\] 
and $u^h(x)\in\R$ denotes the value of $u^h\in\R^{|\Omega^h|}$ at the grid point $x\in\Omega^h$. More precisely, if \eqref{Eq:Reg:FDM} has a solution $u^h\in\R^{|\Omega^h|}$, then there exists $u_\theta\in\mcH$ solving \eqref{Eq:Regularization:FullyDiscrete} such that $u_\theta$ and $u^h$ coincide on the grid $\Omega^h$, i.e., $u_\theta(x) = u^h(x)$ for all $x\in\Omega^h$ \cite[Theorem 4.9]{LangerBehnamian:24}.

\begin{remark}[Distance to the true solution]
Let $H$ as in \cref{Thm:Regularization:FullyDiscrete:Nonunique}, $\hat{u}_\theta \in \argmin_{u_\theta\in H} \mcJ_{\reg}^{h,\theta}(u_\theta)$ and $\hat{u}^h\in\argmin_{u^h\in\R^{|\Omega^h|}} \mcJ_{\reg}^{\FD}(u^h)$ such that $\hat{u}_\theta(x) = \hat{u}^h(x)$ for all $x\in\Omega^h$. In the proof of \cref{Thm:Regularization:FullyDiscrete:Nonunique} we construct a nontrivial $\Phi\in H$ that vanishes on $\Omega^h$ such that $\hat u_\theta + \lambda\Phi\in H$ is a minimizer of $\mcJ_{\reg}^{h,\theta}$ for every $\lambda\in\R$. 
Let $u^*$ denote a solution of \eqref{Eq:Regularization:ContProblem} restricted (or sampled) to $\Omega^h$. Then we obtain 
\[
\sum_{x\in\Omega^h} |\hat u_\theta(x)+\lambda\Phi(x)-u^*(x)|
= \sum_{x\in\Omega^h} |\hat u_\theta(x)-u^*(x)| 
= \sum_{x\in\Omega^h} |\hat{u}^h(x)-u^*(x)| \qquad\text{for all }\lambda\in\R.
\]
Thus, while \eqref{Eq:Regularization:FullyDiscrete} exhibits the same non-uniqueness in $H$ as \eqref{Eq:Regularization:QuadratureLoss}, this non-uniqueness does not alter the discrete finite-difference solution on the stencil.
\end{remark}

As in the Deep Ritz method, here neural networks may equivalently be taken as globally defined functions $f_\theta:\mathbb R^d \to \mathbb R$, since the neural network discrete loss functionals only probe their values on $\Omega^h$; see \cref{Rem:Omega}.

\section{Weak PINN}\label{Sec:wPINN}

We consider an abstract wPINN discretization of a variational problem. Let $U$ be a Hilbert space and let $V_h \subset U$ be a finite-dimensional test space with $\dim(V_h)=n<\infty$. The admissible trial space is denoted by $\mcA \subset U$ (e.g., incorporating essential boundary conditions), and the admissible neural trial class by \(\mcH\subseteq \mcA\).
Let $a(\cdot,\cdot): U \times V_h \to \mathbb{R}$ be a bilinear form and $F: V_h \to \mathbb{R}$ a linear functional.
The wPINN problem reads:
\begin{equation}\label{eq:wPINN}
\text{find } u \in \mcH \text{ such that } a(u,v_h) = F(v_h) \quad \forall v_h \in V_h.
\end{equation}
While the potential ill-posedness of such wPINN formulations has been anticipated in previous work, including \cite{BeCaPi2022}, a rigorous, architecture-independent proof at the level of the exact weak formulation has so far been missing, to the best of our knowledge. The results below provide such a proof.

\subsection{Non-Uniqueness of Solutions}

For the non-uniqueness argument, we introduce a homogeneous neural perturbation class \(\mcH_0\subset U\). We assume that \(\mcH_0\) is closed under finite linear combinations, contains at least \(n+1\) linearly independent functions, and preserves admissibility of the neural trial class in the sense that
\[
u\in\mcH,\ \Phi\in\mcH_0 \quad\Longrightarrow\quad u+\Phi\in\mcH .
\]

We begin with the homogeneous case $F\equiv 0$, corresponding to a zero right-hand side in the weak formulation.

\begin{theorem}[Non-uniqueness of homogeneous case]
\label{Thm:wPINN:HomNonuniq}
The homogeneous problem
\begin{equation}\label{eq:wPINN_hom}
\text{find } u\in\mcH_0 \text{ such that } a(u,v_h)=0 \quad \forall v_h\in V_h
\end{equation}
admits nontrivial solutions. In particular, there exists $\Phi \in \mcH_0 \setminus \{0\}$ such that
\[
a(\Phi,v_h) = 0 \quad \forall v_h \in V_h.
\]
Consequently, since \(0\in\mcH_0\), the homogeneous problem is not unique in \(\mcH_0\).
\end{theorem}

\begin{proof}
Let $\{\phi_i\}_{i=1}^n$ be a basis of $V_h$ and define the linear operator
\[
T:\mcH_0 \to \mathbb R^n, \qquad T(u):=\bigl(a(u,\phi_1),\dots,a(u,\phi_n)\bigr).
\]
By assumption on \(\mcH_0\), there exist linearly independent functions \(u_0,\dots,u_n\in\mcH_0\). Then the $n+1$ vectors $T(u_0),\dots,T(u_n)\in\mathbb R^n$ are linearly dependent, so there exist coefficients $\lambda_0,\dots,\lambda_n\in\R$, not all zero, such that
\[
\sum_{j=0}^n \lambda_j\, T(u_j)=0.
\]
Since $\mcH_0$ is closed under finite linear combinations, $\Phi:=\sum_{j=0}^n \lambda_j u_j \in \mcH_0$. Then $T(\Phi)=0$, i.e., $a(\Phi,\phi_i)=0$ for all $i=1,\dots,n$. Since \(\{\phi_i\}_{i=1}^n\) is a basis of \(V_h\) and \(a(\Phi,\cdot)\) is linear, this implies \(a(\Phi,v_h)=0\) for all \(v_h\in V_h\). Moreover $\Phi\not\equiv 0$ since $u_0,\dots,u_n$ are linearly independent and not all $\lambda_j$, $j=0,\ldots,n$, vanish.
\end{proof}

\begin{corollary}[Non-uniqueness of inhomogeneous problem]
\label{Cor:wPINN:InhomNonuniq}
If problem \eqref{eq:wPINN} admits at least one solution $u^*\in\mathcal H$, then it admits infinitely many solutions in $\mcH$. More precisely, $u^* + \lambda \Phi \in \mcH$ is a solution for every $\lambda\in\R$, where $\Phi\in\mcH_0\setminus\{0\}$ solves \eqref{eq:wPINN_hom}.
\end{corollary}

\begin{proof}
Let $\Phi\in\mcH_0\setminus\{0\}$ be as in \Cref{Thm:wPINN:HomNonuniq}. By the admissibility-preservation assumption on \(\mcH_0\), we have \(u^*+\lambda\Phi\in\mcH\) for all \(\lambda\in\R\). Then for all $v_h\in V_h$, 
\[
a(u^*+\lambda \Phi,v_h)=a(u^*,v_h) + \lambda a(\Phi,v_h)=F(v_h)+0=F(v_h),
\]
so $u^*+ \lambda \Phi$ is a solution for any $\lambda\in\R$.
\end{proof}

\Cref{Thm:wPINN:HomNonuniq} and \cref{Cor:wPINN:InhomNonuniq} show that non-uniqueness of the weak PINN formulation arises whenever the neural network architectures are not fixed and are allowed to increase in width, so that linear combinations of finitely many realizable networks remain realizable by a wider network.
Similar non-uniqueness and instability phenomena in weak neural discretizations have also been observed in the numerical analysis literature. In particular, in \cite{BeCaPi2022} it is demonstrated that, for wPINN formulations of elliptic problems, spurious zero-residual solutions may exist and that stability requires a suitable inf-sup condition between trial and test spaces. Their analysis motivates stabilization strategies based on projection or interpolation of the neural trial functions, which aligns with the abstract non-uniqueness mechanism identified in \Cref{Thm:wPINN:HomNonuniq}. 
More precisely, in the wPINN framework of \cite{BeCaPi2022}, the differential operators appearing in the weak formulation are not applied to the neural network function itself. Instead, the neural network is first projected (or interpolated) into a finite element space, and all spatial derivatives are taken with respect to this polynomial finite element representation, thereby avoiding the use of pointwise neural-network derivatives (e.g., via automatic differentiation) at the PDE level.

Importantly, the non-uniqueness in the wPINN framework is a property of the discrete weak formulation and is independent of the well-posedness of the underlying continuous variational problem:
\begin{equation*}
\text{find } u \in \mcA \text{ such that } a(u,v) = F(v) \quad \forall v \in U.
\end{equation*}
We emphasize that the non-uniqueness of the wPINN formulation arises already at the level of the exact variational formulation and is not a consequence of numerical quadrature, in contrast to the situation for PINNs, cf.\ \cite{Langer:26}.

\begin{remark}[Numerical integration in wPINNs]
Although the wPINN formulation \eqref{eq:wPINN} is posed in terms of weak derivatives and exact integrals, any practical implementation necessarily requires numerical quadrature to approximate the bilinear form $a(\cdot,\cdot)$ and the functional $F$. As a consequence, the resulting quadrature-based formulation depends only on finitely many pointwise evaluations of the neural network and, depending on the problem, its classical derivatives at quadrature points.

Thus, while the continuous weak formulation already admits non-unique solutions due to the finite-dimensionality of the test space $V_h$, numerical integration additionally converts the problem into a finite-information restricted system of the type discussed in \cref{Sec:FiniteLoss}. 
In particular, quadrature does not restore uniqueness; if anything, it may introduce further null directions corresponding to functions that vanish at all quadrature points while remaining nonzero elsewhere, analogous to those observed for Deep Ritz methods and neural network discretizations of variational regularization.
\end{remark}

\section{Conclusion}\label{Sec:Conclusion}

In this work, we showed that several neural network-based discretizations of variational problems, including the Deep Ritz method, variational regularization approaches, and wPINN formulations, can be ill-posed at the discrete level even when the corresponding continuous problems are well-posed. In each case, the discrete formulation admits non-uniqueness or loss of control that is not reflected in the underlying variational model.

A common structural reason for this behavior is the reliance on finite-information discretizations together with pointwise measurements. When losses or formulations depend only on finitely many measurements -- through quadrature, collocation, or a finite-dimensional test space -- the resulting discrete problem cannot uniquely determine global behavior. 
This issue becomes particularly pronounced when derivatives are evaluated only pointwise, as is typically the case when using automatic differentiation. When such evaluations are restricted to finitely many locations, they do not provide sufficient information to control the behavior of the function over the domain.

In energy-based formulations and variational regularization methods, this mismatch implies that the discrete loss can admit infinitely many distinct minimizers that are indistinguishable by the finite measurements entering the loss. In particular, one can construct families of minimizers that agree on the sampled points but differ elsewhere in the domain. In derivative-only formulations, the degeneracy may be even larger because the point values themselves need not be controlled.
In regularization problems, this can lead to infinitely many zero-loss solutions, so that the intended regularizing effect is not realized at the discrete level. PINNs based on pointwise derivatives exhibit the same structural difficulty, as shown in \cite{Langer:26}, where ill-posedness arises from pointwise enforcement of continuous differential operators combined with finite-information discretizations. In wPINN formulations, a related loss of uniqueness arises already from the finite dimensionality of the test space, without requiring any optimization, and numerical quadrature does not remove this ambiguity.

When derivatives are approximated by finite differences, ill-posedness in the neural ansatz space persists, but all neural minimizers coincide on the finite-difference stencil with a classical finite-difference solution. At the level where information is prescribed, the discrete problem therefore recovers the behavior of a standard discretization, in contrast to formulations based on pointwise derivatives.

Taken together, these results indicate that the observed difficulties are not caused by neural networks themselves, but by the use of pointwise differentiation in discrete settings that encode only finite information. While pointwise derivatives are readily available in modern implementations, their use calls for caution and may need to be aligned with a discretization that meaningfully constrains the resulting expressions. Without such alignment, discrete neural variational methods may fail to reflect the behavior of the underlying continuous problem, despite small or vanishing discrete loss values.

\appendix
\section{Interpolation Lemmas for Networks}\label{Sec:ReLUInterpolation}

Here we present auxiliary results needed to prove \cref{Thm:Regularization:Existence:NonUniqueness}.
\begin{lemma}\label{Lem:ReLU:InterpolationNN}
Let $\Omega^h :=\{x_i\}_{i=1}^N \subset \Omega \subset \R^d$ be a finite set of pairwise distinct points, and $g^h: \Omega^h \to \R$. For any $L \ge \lceil \log_2(d+1)  \rceil + 1$ there exists a ReLU neural network $\Phi:\mathbb{R}^d\to\mathbb{R}$ of depth $L$ such that 
\[
\Phi(x) = g^h(x) \quad x\in \Omega^h \quad\text{and}\quad D^\beta \Phi(x)=0 \quad x\in\Omega^h, \ |\beta|> 0.
\]
\end{lemma}

\begin{proof}
Choose pairwise disjoint polyhedral open sets $\mcU_i\subset \Omega$ with $x_i \in \mcU_i$ and pick open polyhedral sets $\mcV_i$ such that $x_i\in \mcV_i \Subset \mcU_i$ for $i=1,\ldots,N$. For each $i\in\{1,\ldots,N\}$, choose a finite triangulation of the polyhedral set $\mcU_i \setminus \overline{\mcV}_i$. Define $\psi_i$ to be affine on each simplex on the triangulation, set $\psi_i\equiv 1$ on $\mcV_i$ and $\psi_i\equiv 0$ on $\R^d\setminus \mcU_i$, and choose the affine pieces so that $\psi_i$ extends continuously to $\R^d$.
This yields continuous and piecewise affine functions $\psi_i:\R^d\to\R$ for $i=1,\ldots,N$.
Then $\Psi (x) := \sum_{i=1}^N g^h(x_i) \psi_i(x)$ yields a continuous and piecewise affine function with $\Psi (x) = g^h(x_i)$ for $x\in \mcV_i$ and $i=1,\ldots,N$.

Since any continuous and piecewise affine function is representable by a ReLU neural network of depth $L_0 := \lceil \log_2(d+1)  \rceil + 1$ \cite[Theorem 2.1]{ArBaMiMu:16}, $\Psi$ is realized by a ReLU neural network.
Because $\Psi$ is constant on each $\mcV_i$, $i=1,\ldots,N$, all derivatives vanish at $x_i$.

To obtain a depth-$L$ neural network with $L\geq L_0$ we just insert layers that implement the identity mapping. This can be realized by $f(t):=\sigma(t)-\sigma(-t) = t$ for all $t\in\R$  and composing the final neural network as $\Phi = f \circ f \circ \cdots \circ f \circ \Psi$ by using $L-L_0$ $f$'s. This proves the statement. 
\end{proof}

\begin{lemma}\label{Lem:Smooth:InterpolationNN}
Let \(\Omega^h :=\{x_i\}_{i=1}^N \subset \Omega \subset \mathbb R^d\) be a finite set of pairwise distinct points, and let \(m\in\mathbb N\) and \(L\geq 2\). 
Set \(\ell:=N(m+1).\)
Assume that \(\sigma\in C^{\ell-1}(\mathbb R)\) and that there exists \(s\in\mathbb R\) such that
\[
D^k\sigma(s)\neq 0 \qquad \text{for all }0\le k\le \ell-1 .
\]
If \(L>2\), assume in addition that \(\sigma\) is strictly monotone. 
Let \(g^h:\Omega^h\to\mathbb R\) be prescribed. Then there exists a depth-\(L\) neural network \(u_\theta\) with hidden layer widths \(d_j=1\), \(j=1,\ldots,L-2\), and \(d_{L-1}=\ell,\) such that
\[
u_\theta(x_i)=g^h(x_i), \qquad D^\beta u_\theta(x_i)=0 \quad\text{for all }1\le |\beta|\le m,\ i=1,\ldots,N .
\]
\end{lemma}
\begin{proof}
We distinguish two cases.

\medskip
\noindent\textbf{Case \(L=2\).}
In this case the statement follows directly from \cite[Corollary~A.3]{Langer:26} applied to the finite set \(\Omega^h=\{x_i\}_{i=1}^N\), with prescribed Hermite data
\[
u_\theta(x_i)=g^h(x_i), \qquad D^\beta u_\theta(x_i)=0 \quad\text{for }1\le |\beta|\le m .
\]
Indeed, with \(\ell=N(m+1)\), the assumptions on \(\sigma\) are precisely those required by \cite[Corollary~A.3]{Langer:26}. This gives a one-hidden-layer, hence depth-\(2\), neural network with \(d_1=\ell\) hidden units.

\medskip
\noindent\textbf{Case \(L>2\).}
First, choose \(v\in\mathbb R^d\) such that \(v\cdot x_i\neq v\cdot x_j\) for all \(i\neq j\). Indeed, for each pair \(i\neq j\), the set
\[
A_{ij}:=\left\{v\in\mathbb R^d:\ v\cdot(x_i-x_j)=0\right\}
\]
is a codimension-one hyperplane, and finitely many proper hyperplanes cannot cover \(\mathbb R^d\), cf.\ \cite[Remark A.5]{Langer:26}.

Define \(\varphi_0(x):=v \cdot x\) and for \(i=1,\ldots,L-2\) define recursively \(\varphi_i(x):=\sigma(\varphi_{i-1}(x)).\)
Set \(f:=\varphi_{L-2}.\)
Since \(\sigma\) is injective and \(\varphi_0(x_i)\neq \varphi_0(x_j)\) for \(i\neq j\), induction gives
\[
f(x_i)\neq f(x_j) \qquad \text{for all }i\neq j .
\]
Thus the scalar nodes \(t_i:=f(x_i),\ i=1,\ldots,N,\) are pairwise distinct.

We now apply \cite[Corollary~A.3]{Langer:26} in dimension one, i.e., $d=1$, to the nodes \(t_1,\ldots,t_N\). We prescribe
\[
\Phi(t_i)=g^h(x_i),\qquad \text{ and } \qquad \frac{d^k \Phi}{dt^k}(t_i)=0 \quad\text{for }1\le k\le m,\ i=1,\ldots,N .
\]
Since the interpolation is one-dimensional, i.e., $d=1$, and involves \(N\) nodes with derivatives up to order \(m\), we have \(\ell=N(m+1).\)
Thus \cite[Corollary~A.3]{Langer:26} applies under the assumptions
\[
\sigma\in C^{\ell-1}(\mathbb R), \qquad D^{(k)}\sigma(s)\neq 0
\]
for $0\le k\le \ell-1$ and some $s\in\R$. 
Therefore there exists a one-hidden-layer scalar neural network \(\Phi:\mathbb R\to\mathbb R\) with the above interpolation properties and $\ell$ hidden units.

Defining \(u_\theta(x):=\Phi(f(x)),\) we obtain for all $i=1,\ldots,N$ that
\[
u_\theta(x_i)=\Phi(f(x_i))=\Phi(t_i)=g^h(x_i).
\]
Moreover, for every multi-index \(\beta\) with \(1\le |\beta|\le m\), the Faà di Bruno formula gives \(D^\beta(\Phi\circ f)(x_i)\) as a finite sum of terms involving \(\frac{d^k \Phi}{dt^k}(f(x_i)),\ 1\le k\le |\beta|,\) multiplied by derivatives of \(f\) at \(x_i\) for all $i=1,\ldots,N$. Since \(f(x_i)=t_i\) and \(\frac{d^k\Phi}{dt^k}(t_i)=0\) for \(1\le k\le m,\) all such terms vanish. 
Hence \(D^\beta u_\theta(x_i)=0\) for all \(1\le |\beta|\le m\), \(i=1,\ldots,N.\)

Finally, \(u_\theta=\Phi\circ f\) is a depth-\(L\) neural network, since \(f\) occupies the first \(L-2\) hidden layers, each of width one, and \(\Phi\) uses the last hidden layer, of width $\ell$, together with the final linear output layer.
\end{proof}

\bibliographystyle{abbrv}
\bibliography{nmh, PINN_Ref}

\end{document}